\documentclass[draft]{amsart}

\usepackage{verbatim, amssymb, enumerate, mathtools}

\renewcommand \a{\alpha}
\renewcommand \b{\beta}

\newcommand \n{\nabla}
\newcommand \la{\lambda}

\newcommand \br{\mathbb{R}}

\newcommand \Span{\operatorname{Span}}

\newcommand \so{\mathfrak{so}}

\newcommand \og{\mathfrak{so}}

\newcommand\g{\mathfrak g}
\newcommand\h{\mathfrak h}
\newcommand\m{\mathfrak m}

\newcommand\f{\mathfrak f}
\newcommand\ig{\mathfrak i}
\newcommand\slg{\mathfrak{sl}}

\newcommand\Lg{\mathfrak l}
\newcommand\sg{\mathfrak s}

\newcommand \Sym{\operatorname{Sym}}

\newcommand \ad{\operatorname{ad}}

\newcommand \<{\langle}
\renewcommand \>{\rangle}
\newcommand \ip{\< \cdot, \cdot \>}

\newcommand \mU{\mathcal{U}}
\newcommand \cR{\mathcal{R}}

\theoremstyle{plane}
\newtheorem{theorem}{Theorem}
\newtheorem*{theorem*}{Theorem}
\newtheorem*{corollary*}{Corollary}
\newtheorem*{conj*}{Conjecture}
\newtheorem{lemma}{Lemma}

\newtheorem*{prop*}{Proposition}

\theoremstyle{definition}

\newtheorem*{definition*}{Definition}

\theoremstyle{remark}
\newtheorem*{remark*}{Remark}
\newtheorem{example}{Example}

\makeatletter
\renewcommand{\p@enumii}{\theenumi)(}
\makeatother

\begin{document}

\title{Totally geodesic hypersurfaces of homogeneous spaces}

\author{Y.Nikolayevsky}
\address{Department of Mathematics and Statistics, La Trobe University, Melbourne, Victoria, 3086, Australia}
\email{y.nikolayevsky@latrobe.edu.au}

%\date{\today}

%\thanks{Supported by DGS grant}

 %53B25 Local submanifolds; 53C40 Global submanifolds; 53C30 Homogeneous manifolds

\begin{abstract}
We show that a simply connected Riemannian homogeneous space $M$ which admits a totally geodesic hypersurface $F$ is isometric to either (a) the Riemannian product of a space of constant curvature and a homogeneous space, or (b) the warped product of the Euclidean space and a homogeneous space, or (c) the twisted product of the line and a homogeneous space (with the warping/twisting function given explicitly). In the first case, $F$ is also a Riemannian product; in the last two cases, it is a leaf of a totally geodesic homogeneous fibration. Case (c) can alternatively be characterised by the fact that $M$ admits a Riemannian submersion onto the universal cover of the group $\mathrm{SL}(2)$ equipped with a particular left-invariant metric, and $F$ is the preimage of the two-dimensional solvable totally geodesic subgroup. %of $\widetilde{\mathrm{SL}(2)}$
\end{abstract}

\maketitle

\section{Introduction}
\label{s:intro}

% vsyo perestavit': post krivizny, 2dim solv, sl2. main th: either first or fibr, riem submersion
% one obv example which comes to mind is const... And indeed Tojo... proved ...
% two others are metric Lie groups
% first solv... (geometrically: hyperbolic space and a geodesic)

The study of totally geodesic submanifolds of homogeneous spaces dates back to the classical result of \'{E}lie Cartan from 1927 (\cite{C} or \cite[IV, \S 7]{Hel}), which says that a totally geodesic submanifold of a symmetric space is the exponent of a Lie triple system. Homogeneous totally geodesic submanifolds of nilpotent Lie groups have been extensively studied in \cite{Ebe, KP, CHN1, CHN2}. The classification of totally geodesic submanifolds of nonsingular two-step nilpotent Lie groups is given in \cite{Ebe}.

In the last two decades, a remarkable progress has been achieved in the study of one-dimensional totally geodesic submanifolds --- \emph{homogeneous geodesics} (the geodesics which are the orbits of a one-dimensional isometry group); this includes the deep existence results \cite{Kai, KS, Dus} and the investigation of the \emph{g.o. spaces} --- homogeneous spaces all of whose geodesics are homogeneous (see e.g. \cite{Gor, AN}).

In this paper we investigate the other extremity --- totally geodesic hypersurfaces (not necessarily homogeneous) of homogeneous spaces. As one may expect, the existence of such a hypersurface imposes strong restrictions on the ambient space. In particular, if a homogeneous space admits a totally geodesic hypersurface, then it must be a space of constant curvature, provided it belongs to one of the following classes: irreducible symmetric spaces \cite{CN}, normal  homogeneous spaces \cite{To2}, and more generally, naturally reductive homogeneous spaces \cite{Ts1, To1}. Totally geodesic hypersurfaces and extrinsic hyperspheres in manifolds with special holonomy have been recently studied in \cite{JMS}. By \cite[Proposition~5]{CHN1},
if a nilmanifold admits a totally geodesic homogeneous hypersurface $F$, then its metric Lie algebra is the direct orthogonal sum of a one-dimensional ideal and the ideal tangent to $F$.

We prove the following classification theorem. %everywhere: of type(c) rather than "from" or "belonging to" case (c)?

\begin{theorem}\label{t:tg}
Suppose $M$ is a simply connected, connected Riemannian homogeneous space and $F \subset M$ is a complete connected totally geodesic hypersurface. Then one of the following holds. %exactly
\begin{enumerate}[{\rm (a)}]
  \item \label{it:th1}
  $M=M_1(c) \times M_2$, the \textbf{Riemannian product} of a space $M_1(c)$ of constant curvature $c$ and a homogeneous space $M_2$.
  The hypersurface $F$ is the product $F_1(c) \times M_2$, where $F_1(c) \subset M_1(c)$ is totally geodesic.

  \item \label{it:th2}
  $M=\br^m \prescript{}{f}{\times} M_2$, the \textbf{warped product} of $\br^m, \, m >0$, and a homogeneous space $M_2=G/H$, with the warping function $f: M_2 \to \br$ defined by $f(gH)=\chi(g)$, where $\chi: G \to (\br^+, \cdot)$ is a nontrivial homomorphism with $\chi(H)=1$.
  The hypersurface $F$ is the Cartesian product of a hyperplane $\br^{m-1} \subset \br^m$ and $M_2$.

  \item \label{it:th3}
  $M=\br \prescript{}{f}{\times} M_2$, the \textbf{twisted product} of $\br$ and a homogeneous space $M_2$. The hypersurface $F$ is a leaf of the totally geodesic fibration $\{t\}\times M_2, \; t \in \br$.

  Moreover, the curves $\br \times \{x\}, \; x \in M_2$, are congruent helices of order two with the curvature $k$ and the torsion $\kappa \ne 0$. With a particular choice of local coordinates $t$ on $\br$ and $u$ on $M_2$, the twisting function is given by $f(t,u)=(\sinh(\a(u))\cos(\kappa t+\b(u)) +\cosh(\a(u)))^{-2}$, where locally $\a, \b : M_2 \to \br$ satisfy $\|\n \a \|^2 = \sinh(\a)^2 \|\n\b\|^2= k^2$.
\end{enumerate}
\end{theorem}

The \emph{warped} (the \emph{twisted}) product $M_1 \prescript{}{f}{\times} M_2$ of Riemannian manifolds $(M_1, ds_1^2)$ and $(M_2, ds_2^2)$, with the \emph{warping function} $f: M_2 \to \br^+$ (respectively, with the \emph{twisting function} $f: M_1 \times M_2 \to \br^+$), is the Cartesian product $M_1 \times M_2$ equipped with the metric $f ds_1^2 + ds_2^2$. A smooth curve in a Riemannian space is called a \emph{helix of order $p \ge 0$}, if its first $p$ Frenet curvatures are nonzero constants and the $(p+1)$-st Frenet curvature vanishes (by analogy with curves in $\br^3$, for helices of order two, we call the first two nonzero curvatures the \emph{curvature} and the \emph{torsion}, respectively). Note that we impose the assumption of completeness of $F$ only for convenience; any open portion of a totally geodesic hypersurface of $M$ can be extended to a complete hypersurface by extending all the geodesics.

It follows from Theorem~\ref{t:tg} that apart from Case~\eqref{it:th1}, a totally geodesic hypersurface $F$ is a leaf of a totally geodesic fibration of codimension one.

Theorem~\ref{t:tg} is intentionally stated in a purely ``Riemannian" language (except for a small amount of algebra in Case~\eqref{it:th2}) avoiding the choice of a particular presentation of $M$ as $G/H$. An important question in the theory of totally geodesic submanifolds of homogeneous spaces is when such a submanifold is \emph{homogeneous} (that is, is the orbit of a subgroup of $G$). From Theorem~\ref{t:subm} below (or from the proof of Theorem~\ref{t:tg} given in Section~\ref{s:pf}) one can deduce that in Case~\eqref{it:th3} of Theorem~\ref{t:tg}, the hypersurface $F$ is homogeneous relative to any choice of a connected transitive group $G$ of isometries of $M$. The answer in the other two cases depends on a particular presentation. In Case~\eqref{it:th1} it can easily be in negative: the group $\mathrm{SU}(2)$ with a metric of constant positive curvature contains no two-dimensional subgroups. An example of a non-homogeneous totally geodesic hypersurface from Case~\eqref{it:th2} is given in Section~\ref{s:pf}. Note however that in all the cases, the Riemannian manifold $F$ is homogeneous relative to the induced metric.%, from the intrinsic point of view

Theorem~\ref{t:tg} has the following obvious but useful corollary.
\begin{corollary*}
A compact, simply connected, connected Riemannian homogeneous space that admits a totally geodesic hypersurface $F$ is the Riemannian product of a standard sphere $S^m, \; m \ge 2$, and a compact homogeneous space $M_2$; then $F$ is (a domain of) the product of a great hypersphere $S^{m-1}$ and $M_2$.
\end{corollary*}

One can give an alternative, more algebraic description of the totally geodesic hypersurface from Case~\eqref{it:th3} of Theorem~\ref{t:tg}. The ``smallest" example of such a hypersurface is constructed as follows.

\begin{example}\label{ex:sl2}
Take $M=\widetilde{\mathrm{SL}(2)}$, the universal cover of the group $\mathrm{SL}(2)$. Denote $\g=\slg(2)$. Let $\f \subset \g$ be a two-dimensional subalgebra and let $N \in \g$ span the one-dimensional subalgebra $\og(2) \subset \g$. Up to automorphism and scaling one can choose, in the defining representation of $\g$,
\begin{equation*}
    N=\left(
         \begin{array}{cc}
           0 & 1 \\
           -1 & 0 \\
         \end{array}
       \right), \qquad
    \f=\bigg\{\left(
         \begin{array}{cc}
           x & y \\
           0 & -x \\
         \end{array}
       \right) \, : \, x, y \in \br \bigg\}\, .
\end{equation*}
Introduce an inner product on $\g$ by requiring that $N \perp \f$ and by specifying it further on $\f$ in such a way that the operator $\pi_\f \ad_N \pi_\f$ is skew-symmetric. Explicitly, choose arbitrary nonzero $a, b \in \br$ and define the inner product $\ip$ in such a way that the following basis is orthonormal:
\begin{equation} \label{eq:sl2inner}
    E_1=a N=a \left(
         \begin{array}{cc}
           0 & 1 \\
           -1 & 0 \\
         \end{array}
       \right), \qquad E_2=2 b \left(
         \begin{array}{cc}
           0 & 1 \\
           0 & 0 \\
         \end{array}
       \right), \qquad
    E_3= b \left(
         \begin{array}{cc}
           1 & 0 \\
           0 & -1 \\
         \end{array}
       \right).
\end{equation}
Then $\f$ is a totally geodesic subalgebra of the metric Lie algebra $(\g, \ip)$ \cite[Theorem~7.2]{Ts2}, and so the subgroup $F_1$ tangent to $\f$ is a totally geodesic hypersurface of $M=\widetilde{\mathrm{SL}(2)}$ equipped with the left-invariant metric obtained from the inner product \eqref{eq:sl2inner}. Note that $F_1$ is isometric to the hyperbolic space and the functions $\a, \b$ from Case~\eqref{it:th3} are, up to scaling, the polar coordinates on $F_1$.
\end{example} %it is easy to see that

\begin{theorem}\label{t:subm}
Under the assumptions of Theorem~\ref{t:tg}, either the pair $(M=G/H, F)$ belongs to one of the cases~\emph{\eqref{it:th1}, \eqref{it:th2}}, or otherwise there exists a normal subgroup $N \subset G$ such that $H \subset N, \; G/N \simeq \widetilde{\mathrm{SL}(2)}$, and the projection $\pi:M \to \widetilde{\mathrm{SL}(2)}$ \emph{(}where the metric on $\widetilde{\mathrm{SL}(2)}$ is constructed as in Example~\ref{ex:sl2}\emph{)} is a Riemannian submersion, and $F = \pi^{-1}F_1$.
\end{theorem}

%The paper is organised as follows...

%In very vague terms, the theorem says that apart from two obvious examples of totally geodesic hypersurfaces -- a hyperplane in the space of constant curvature and a leaf of homogeneous totally geodesic fibration of codimension one -- nothing else can happen.

\section{Proofs}
\label{s:pf}

Let $M=G/H$ be a simply connected, connected Riemannian homogeneous space, with $G$ a simply connected, closed, connected transitive group of isometries acting on $M$ from the left and $H$ the (connected) isotropy subgroup of a point $o \in M$. Let $\pi:G \to M$ be the natural projection with $\pi(e)=o$. Denote $\ip, \; \n$ and $R$ the metric, the Levi-Civita connection and the curvature tensor of $M$ respectively. For vector fields $X,Y \in TM$ we define the operator $X \wedge Y \in \so(TM)$ by $(X \wedge Y)Z=\<X,Z\>Y-\<Y,Z\>X$. Denote $\cR \in \Sym(\so(TM))$ the curvature operator, the symmetric operator defined by $\<\cR(X \wedge Y), Z \wedge V\> =\<R(X,Y)Z, V\>$, where the inner product on the left-hand side is the natural inner product on $\so(TM)$. For a vector $X$ and a subspace $V$ we denote $X \wedge V$ the subspace $\Span(X \wedge Y \, : \, Y \in V)$.

Let $F \in M$ be a connected totally geodesic hypersurface. Without loss of generality we can assume that $o \in F$. Moreover, as $M$ is an analytic Riemannian manifold and as $F$ is totally geodesic, hence minimal, it is an analytic submanifold of $M$. Therefore we can (and will) replace $F$ by a small open disc of $F$ containing $o$. Let $\xi$ be a continuous unit vector field normal to $F$. Consider the \emph{Gauss image} of $F$ defined by $\Gamma(F) = \{dg^{-1}\xi(x) \, : \, x \in F, \, g \in G, \, g(o)=x\}$. The set of pairs $(x,g) \in F \times G$ such that $g(o)=x$ is (locally) diffeomorphic to $\pi^{-1}F \simeq F \times H$, so $\Gamma(F)$ is the image of a continuous (in fact, analytic) ``Gauss map" $\Phi: \pi^{-1}F \to S_o(1)$, where $S_o(1)$ is the unit sphere of $T_oM$. As $H$ is connected, $\Gamma(F)$ is also connected. Moreover, $\Gamma(F)$ is $H$-left-invariant. Then the subspace $D_o = \Span(\Gamma(F)) \subset T_oM$ is $H$-left-invariant, as also is its orthogonal complement $D^\perp_o$. Hence we can define two orthogonal complementary $G$-left-invariant distributions $D$ and $D^\perp$ on $M$ such that $D(o)=D_o$ and $D^\perp(o)=D^\perp_o$. Denote $m = \dim D$.

% perenazvat' d dperp? (pomenyat' mestami) ili: H and V?

\begin{lemma}\label{l:ddperp}
In the above notation we have:
\begin{enumerate}[{\rm 1.}]
  \item \label{it:dperp}
  The distribution $D^\perp$ is integrable with totally geodesic leaves. The leaf of $D^\perp$ passing through $o$ locally lies in $F$.
  \item \label{it:d}
  The distribution $D$ is integrable with totally umbilical leaves.
  \item \label{it:curv}
  $D \wedge D$ lies in an eigenspace of $\cR$, so that there exists $\la \in \br$ such that $\cR(X \wedge Y)=\la X\wedge Y$, for all $X,Y \in D$.
\end{enumerate}
\end{lemma}
% extrinsic sphere in 2?

\begin{proof}
1. First note that if $X$ is tangent to $D^\perp$ at some point $x \in F$ and $x=g(o)$, then $dg^{-1}X \in D^\perp_o$ (as $D^\perp$ is $G$-left-invariant), hence $dg^{-1}X \perp dg^{-1}\xi(x)$, so $X \perp \xi(x)$. Thus $D^\perp$ is tangent to $F$.

Now let $X, Y$ be two vector fields tangent to $D^\perp$ in a neighbourhood of $o$. They must be tangent to $F$ at the points of $F$. As $F$ is totally geodesic, we have $(\n_XY)_{|o} \perp \xi(o)$. Moreover, for any $x \in F$ and any $g \in G$ such that $g(o)=x$, the vector field $dgX$ and $dgY$ are tangent to $D$ and to $F$ (at the points of $F$), so $(\n_{dgX}dgY)_{|x} \perp \xi(x)$, hence $(\n_XY)_{|o} \perp dg^{-1}\xi(x)$. It follows that $(\n_XY)_{|o} \in D^\perp_o$. As $D^\perp$ is $G$-left-invariant it follows that everywhere on $G$ we have $\n_XY \in D^\perp$, for any vector fields $X, Y \in D^\perp$. Therefore $[D^\perp,D^\perp] \subset D^\perp$, and the leaves tangent to $D^\perp$ are totally geodesic submanifolds of $M$.

2. Let $\eta \in \Gamma(F)$ and let $g \in G$ and $x \in F$ be chosen in such a way that $\eta=dg^{-1}\xi(x)$. Let $Z' \in T_xF \cap D(x)$ and let $X' \in D^\perp$ be a vector field in a neighbourhood of $x$. Then $(\n_{Z'}X')_{|x} \perp \xi(x)$. Acting by $dg^{-1}$ we obtain that $(\n_{Z}X)_{|o} \perp \eta$ for any $Z \in D_o \cap \eta^\perp$ and for any vector field $X \in D$ in a neighbourhood of $o$. It follows that every $\eta \in \Gamma(F)$ is an eigenvector of the linear operator $L_X$ on $D_o$ defined by $\<L_X N_1, N_2\>=\<(\n_{N_2}X)_{|o},  N_1\>$ ($L_X$ is the adjoint to the Nomizu operator of $X$). As $\Gamma(F)$ is a connected subset of the unit sphere of $T_oM$ spanning $D_o$ we obtain that $L_X$ is proportional to the identity, so that the bilinear form on $D_o \times D_o$ defined by $(Z_1,Z_2) \mapsto \<(\n_{Z_1}X)_{|o}, Z_2\>$ vanishes for all $Z_1 \perp Z_2, \; Z_1, Z_2 \in D_o$. Let $N_1, N_2 \in D$ be orthogonal vector fields in a neighbourhood of $o$ (if $m(=\dim D) =1$, the claim of the assertion is trivial) and let $X \in D^\perp$ be a vector field in a neighbourhood of $o$. At the point $o$ we have $\<\n_{N_1}N_2,X\>=-\<\n_{N_1}X,N_2\>=0$. It follows that $[N_1,N_2] \in D$ for any two orthogonal vector fields $N_1,N_2 \in D$, hence for any such $N_1, N_2$. Then $D$ is integrable and the second fundamental form of the leaves vanishes on any pair of orthogonal vectors. It follows that the second fundamental form (in every direction from $D^\perp$) is proportional to the induced inner product on $D$, hence the leaves are totally umbilical.

3. Let $g \in \pi^{-1}F$ with $x=g(o) \in F$. From the Codazzi equation at $x$ we have $\<R(X,Y)Z,\xi\>=0$, for all $X,Y,Z \in T_xF$. From the symmetries of the curvature tensor it follows that $\<R(\xi,X)Y,Z\>=\<R_{\xi} X, Z\>\<\xi,Y\>-\<R_\xi X, Y\>\<\xi,Z\>$, for all $X,Y,Z \in T_x M$, where $R_\xi:T_x M \to T_x M$ is the Jacobi operator defined by $R_\xi X=R(\xi,X)\xi$. Then $\cR(\xi \wedge X)=(R_\xi X) \wedge \xi$. As $R_\xi$ is symmetric, there exists an orthonormal basis $e_i, \; i = 1, \dots, n-1$, for $T_xF$ such that $\cR(\xi \wedge e_i)= c_i \xi \wedge e_i$, so that the elements $\xi \wedge e_i \in \so(T_xM)$ are the eigenvectors of $\cR \in \Sym(\so(T_xM))$ \cite[Proposition~4.7]{Ts2}. Acting by $dg^{-1}$, we obtain that for every $g \in \pi^{-1}F$ there is a direct orthogonal decomposition $T_oM=\br \Phi(g) \oplus \oplus_{s=1}^{p(g)} L_s(g)$ (where $\Phi$ is the Gauss map and $\Phi(g) \in \Gamma(F)$) such that every subspace $\Phi(g) \wedge L_s(g) \subset \so(T_oM)$ lies in the eigenspace of $\cR \in \Sym(\so(T_oM))$ with the eigenvalue $\la_s(g)$ (here $\la_s(g)$'s are the $c_i$'s without repetitions). Let $\so(T_oM)=\oplus_{a=1}^N V_a$ be the orthogonal decomposition of $\so(T_oM)$ on the eigenspaces of $\cR$, with $\mu_a$ the corresponding eigenvalues. Then every $\la_s(g)$ equals to one of the constants $\mu_a$. As the Jacobi operator $R_\xi$ depends continuously (in fact, analytically) on $g \in \pi^{-1}F$ and all its eigenvalues belong to the finite set $\{\mu_a\}$ we obtain that the number of eigenvalues $p(g)=p$ is constant and up to relabelling, every subspace $\Phi(g) \wedge L_s(g)$ lies in $V_s$. Moreover the dimensions $m_s=\dim L_s(g)$ are constant and the maps $g \mapsto L_s(g)$ are analytic maps from $\pi^{-1}F$ to the Grassmanians $G(m_s,T_oM)$. It follows that for any $g, h \in \pi^{-1}F$ and for any $s \ne l$, we have $\<\Phi(g) \wedge L_s(g), \Phi(h) \wedge L_t(h)\>=0$, so $(\Phi(g) \wedge L_s(g))\Phi(h) \perp L_l(h)$, therefore $(\Phi(g) \wedge L_s(g))\Phi(h) \subset L_s(h)$. Now if $\Phi(h) \not\perp \Phi(g)$, the subspace $(\Phi(g) \wedge L_s(g))\Phi(h)$ has dimension $m_s$, the same as the dimension of $L_s(h)$. So there exists a small enough neighbourhood $\mU \subset \pi^{-1}F$ of $e$ such that for all $g, h \in \mU$ and all $s=1, \dots, p$, we have $(\Phi(g) \wedge L_s(g))\Phi(h) = L_s(h)$, hence $L_s(h) \subset \br \Phi(g) \oplus L_s(g)$. Let $N_s=\Span(L_s(h) \, : \, h \in \mU)$. Then $\dim N_s \ge m_s$ as $\dim L_s(h) =m_s$, and moreover, since $N_s \subset \br \Phi(g) \oplus L_s(g)$, for all $g \in \mU$, we have $\dim N_s \le m_s+1$. So we have two possibilities: either $\dim N_s = m_s$, in which case the subspaces $L_s(h)$ do not depend on $h$: $L_s(h)=N_s$, for all $h \in \mU$; or $\dim N_s = m_s+1$, in which case the subspaces $\br \Phi(g) \oplus L_s(g)$ do not depend on $g$: $\br \Phi(g) \oplus L_s(g)=N_s$, for all $g \in \mU$. But the latter case occurs for no more than one $s=1, \dots, p$. Indeed, if we suppose that $\br \Phi(g) \oplus L_s(g)=N_s$ and $\br \Phi(g) \oplus L_l(g)=N_l$, for all $g \in \mU$ and for some $s \ne l$, then, as $\Phi(g), L_s(g)$ and $L_l(g)$ are mutually orthogonal, we obtain $\br \Phi(g)= N_s \cap N_l$. It follows that $\Phi(g)$ is constant, and so the subspaces $L_s(g) = N_s \cap (\Phi(g))^\perp$ also do not depend on $g$. But then $L_s(g) =\Span(L_s(h) \, : \, h \in \mU) = N_s$, so $\dim N_s = m_s$, a contradiction. So $\dim N_s=m_s+1$ for no more than one $s=1, \dots, p$, and $\dim N_s=m_s$ for all the other $s$.

Now if $\dim N_s=m_s$ for all $s=1, \dots, p$, then the vector $\Phi(g)$ and all the subspaces $L_s(g)$ are constant: $\Phi(g)=\Phi$ and $L_s(g)=N_s$, for all $g \in \mU$, hence for all $g \in \pi^{-1}F$, by analyticity. Then the distribution $D$ is one dimensional and the claim follows trivially.

Otherwise, suppose that $\dim N_1=m_1+1$. Then again $L_s(g)=N_s$, for all $g \in \pi^{-1}F$ and for all $s \ge 2$. We also have $D_o= \Span(\Phi(g) \, : \, g \in \pi^{-1}F)=\Span(\Phi(g) \, : \, g \in \mU)$ by analyticity, so $D_o \subset N_1$. But for any $\g \in \pi^{-1}F$, we have $\Phi(g) \wedge E_1(g) = \Phi(g) \wedge N_1 \subset V_1$ and $V_1$ is the eigenspace of $\cR$ with the eigenvalue $\la_1$. It follows that $\Phi(g) \wedge D_o \subset V_1$, for all $g \in \pi^{-1}F$, hence $D_o \wedge D_o \subset V_1$, as required.
\end{proof}

\begin{proof}[Proof of Theorem~\ref{t:tg}]
Let $\nu \in D^\perp$ be the mean curvature vector field of the totally umbilical foliation on $M$ defined by $D$. As $D$ is $G$-left-invariant, $\nu$ is also $G$-left-invariant.

We consider two cases for $m =\dim D$.

\medskip

\underline{Case 1.} Suppose that $m \, (=\dim D) > 1$. Then from Codazzi equation and from Lemma~\ref{l:ddperp}\eqref{it:curv} we obtain that the $D^\perp$ component of the vector field $\<Z_1,Z_3\>\n_{Z_2}\nu-\<Z_2,Z_3\>\n_{Z_1}\nu$ vanishes, for any $Z_1, Z_2$ and $Z_3$ tangent to $D$. It follows that $\n_Z \nu$ is tangent to $D$, for any $Z$ tangent to $D$, hence the leaves of the foliation defined by $D$ are \emph{extrinsic spheres}.

We can introduce analytic local coordinates $v^1, \dots, v^m, u^1, \dots , u^{n-m}$ in a neighbouhood of any point $x \in M$ in such a way that $D= \Span (\partial/\partial v^\alpha \, : \, \alpha=1, \dots, m), \; D^\perp= \Span (\partial/\partial u^i \, : \, i=1, \dots, n-m)$. The metric of $M$ is given by $ds^2= A'_{\alpha \beta} (u,v) dv^\alpha dv^\beta + B'_{ij}(u,v) du^i du^j$. As $D^\perp$ is totally geodesic, we obtain that $B'_{ij}=B_{ij}(u)$. From the fact that $D$ is totally umbilical we get $A'_{\alpha \beta}(u,v)=f(u,v) A_{\alpha \beta}(v)$ for some positive analytic function $f$. Then $\nu= -\frac12 B^{ij} \partial (\ln f)/\partial u^i\partial/\partial u^j$ and the fact that the leaves tangent to $D$ are extrinsic spheres gives $\partial^2 (\ln f)/\partial u^i\partial v^\a=0$. It follows that $f$ is a product of a function of the $u$'s by a function of the $v$'s, so (with a slight change of notation) $ds^2= f(u) A_{\alpha \beta} (v) dv^\alpha dv^\beta + B_{ij}(u) du^i du^j$, hence $M$ is locally a warped product. Moreover, from Gauss equation and from Lemma~\ref{l:ddperp}\eqref{it:curv}, every leaf tangent to $D$ has a constant curvature in the induced metric. But the isometry of $M$ which maps a point $x=(u,v)$ to a point $y=(u',v)$ on the same leaf is a homothecy with the coefficient $f(u')/f(u)$. It follows that either $f$ is constant or every leaf tangent to $D$ is flat in the induced metric.

Now, if $f$ is a constant, then $M$ is locally a Riemannian product. As $M$ is simply connected, by de Rham Theorem, it is the Riemannian product of a leaf $M_1$ tangent to $D$ and a leaf $M_2$ tangent to $D^\perp$, with both $M_1$ and $M_2$ homogeneous (as $D$ and $D^\perp$ are $G$-left-invariant). Moreover, by Lemma~\ref{l:ddperp}\eqref{it:curv}, $M_1$ has a constant curvature $c$. Let $F_1$ be (the unique complete) totally geodesic hypersurface of $M_1(c)$ whose normal vector at $o$ is $\xi(o)$. Then the hypersurface $F'=F_1 \times M_2$ is totally geodesic and $F$ is an open subset of $F'$ (as a totally geodesic submanifold is locally uniquely determined by its tangent space at a single point). This gives \textbf{Case~(\ref{it:th1}) of Theorem~\ref{t:tg}}.

Now suppose that $f$ is not a constant. From the above and by \cite[Theorem~A]{BH} the manifold $M$ is a global warped product, $M=\br^m \prescript{}{f}{\times} M_2$, where $M_2$ is the leaf tangent to $D^\perp$ passing through $o$ and $f:M_2 \to \br^+$. The isotropy subgroup $G_2 \subset G$ of $M_2$ acts transitively and isometrically on $M_2$, so $M_2$ is the homogeneous space $G_2/H$. Moreover, every $g \in G_2$ acts on the $\br^m$ fibers by the homothecy with the coefficient $f(g(u))/f(u)$. As this ratio must not depend of $u \in M_2$ we obtain that $f(gH)=\chi(g)$, where $\chi: G_2 \to (\br^+, \cdot)$ is a homomorphism with $\chi(H)=1$. Let $\br^{m-1}$ be the hyperplane of $\br^m$ passing through $o$ and orthogonal to $\xi(o)$. Then the hypersurface $F \subset M$, the (Cartesian) product of $\br^{m-1}$ and $M_2$ is totally geodesic and is (the unique complete) totally geodesic hypersurface of $M$ whose normal vector at $o$ is $\xi(o)$. This gives \textbf{Case~(\ref{it:th2}) of Theorem~\ref{t:tg}}.

\medskip

\underline{Case 2.} Suppose that $m \, (=\dim D)=1$. Let $\tau$ be a unit vector field on $M$ which spans $D$ (so that $\tau(o)=\xi(o)$. By construction, $\tau$ is $G$-left-invariant. Moreover, from \cite[Theorem~A]{BH} the manifold $M$ is diffeomorphic to $\br \times M_2$, where $M_2$ is the leaf of $D^\perp$ passing through $o$. The leaf $M_2$ is a totally geodesic hypersurface and $F$ is an open connected subset of $M_2$. Let  $G_1 \subset G$ be the connected isotropy subgroup of $M_2$. Then $H \subset G_1$ and $G_1$ has codimension one in $G$. It follows that $F$ is a homogeneous totally geodesic hypersurface.

Now, if the vector field $\tau$ is geodesic, then we get back to case Case~\eqref{it:th1}, with $M_1(c)$ a Euclidean line. Furthermore, if $\tau$ is not geodesic, but the leaves of $D$ are ``circles" (one-dimensional extrinsic spheres), that is, if $\n_\tau \nu = -\|\nu\|^2\tau$, then repeating the above arguments we get to Case~\eqref{it:th2}, with $m=1$.

Suppose that $\n_\tau \nu \not \parallel \tau$. Consider the Frenet frame $\tau, \nu_1, \nu_2, \dots$ of the one-dimensional leaves of $D$. We have $\n_\tau \tau =k_1 \nu_1 (=\nu), \; \n_\tau \nu_1 = -k_1 \tau+ k_2 \nu_2$. By the $G$-left-invariancy, all the Frenet curvatures $k_1, k_2, \dots$ are constant; by our assumption, at least the first two of them, $k_1$ and $k_2$, are nonzero.

Passing to the level of Lie algebras, we need to exercise a certain caution, as the standard identification procedure is carried out via Killing vector fields, however the vector fields $\tau, \nu_1, \nu_2, \dots$ are not in general Killing. Denote their values at $o$ by the corresponding Roman letters, so that $T=\tau(o), \; N_1=\nu_1(o), \; N_2=\nu_2(o)$, etc. Note that the spans of the vectors $T, N_1, N_2,\dots$ are one-dimensional $H$-submodules of $T_oM$. Let $\g, \g_1$ and $\h$ be the Lie algebras of $G, G_1$ and $H$ respectively. We have $\h \subset \g_1 \subset \g$, with $\g_1$ a subalgebra of codimension one in $\g$. Choose and fix an $\ad(H)$-invariant complement $\f$ to $\h$ in $\g_1$. The corresponding Killing vector fields are tangent to $D^\perp$, and we can identify $\f$ with $D^\perp(o)$. As the inner product is $\ad(H)$-invariant, we can find a one-dimensional $\ad(H)$-invariant complement to $\g_1$ in $\g$ spanned by an element whose corresponding Killing vector field at $o$ equals $T$. We can identify that element with $T$, and the space $\m=\br T \oplus \f \subset \g$, with $T_oM$. Then we obtain % and ad\h skew?
\begin{equation}\label{eq:tau}
\m=\br T \oplus \f, \qquad T \perp \f, \qquad \g_1=\f \oplus \h, \qquad [T, \h]=0, \qquad [\h, \f] \subset \f, \qquad [\g_1, \g_1] \subset \g_1.
\end{equation}

We have the following lemma:

\begin{lemma}\label{l:helix}
{\ }

\begin{enumerate}[{\rm (a)}]
  \item \label{it:helix1}
  The leaves of $D$ are helices of order two: their first and second Frenet curvatures are nonzero constants, and the third Frenet curvature is zero.

  \item \label{it:helix2}
  Denote $\Lg=\Span(T, N_1, N_2), \; \sg=\Span(N_1, N_2)$ and $\mathfrak{I}=(\m \cap \Lg^\perp) \oplus \h$. Then
  \begin{enumerate}[{\rm (i)}]
    \item \label{it:helix2i}
    $\mathfrak{I}$ is an ideal of $\g$ containing $\h$, with $\g/\mathfrak{I}\simeq\slg(2)$.
    \item \label{it:helix2ii}
    $\g_1=\sg \oplus \mathfrak{I}$.
    \item \label{it:helix2iii}
    $[\Lg,\h]=0$. % nuzhno? posle lemmy remark: mozhno tak vybrat' g_1; esli ubrat' - ob`edinit' (i) i (ii)
  \end{enumerate}

  \item \label{it:helix3}
  The subspace $\Lg \subset \g$, with the induced inner product and with the Lie algebra structure of $\g/\mathfrak{I}$, is isometrically isomorphic to $\slg(2)$ with the metric \eqref{eq:sl2inner}, with $\sg \subset \Lg$ the totally geodesic solvable subalgebra defined by $\sg =\Span(E_2, E_3)$.
\end{enumerate}

\end{lemma}

\begin{proof} (a and b) From the fact that $\f$ is tangent to a totally geodesic hypersurface (and that $\g_1 \subset \g$ is a subalgebra) we obtain that for all $X, Y \in \f$,
\begin{equation}\label{eq:tgf}
\<[T, X]_\m,Y\> + \<[T, Y]_\m,X\>=0,
\end{equation}
where the subscript $\m$ denotes the $\m$-component. To compute the Frenet frame we shall use the following fact: if $\tilde Y$ and $\tilde Z$ are $G$-left-invariant vector fields on $M$ with $\tilde Y(o)=Y, \; \tilde Z(o)=Z$, then
\begin{equation}\label{eq:derlivf}
(\n_{\tilde Y}\tilde Z)(o)=\frac12[Y,Z]_\m+U(Y,Z), \quad\text{where } 2 \<U(Y,Z),X\>=\<[X,Y]_\m,Z\>+\<[X,Z]_\m,Y\>,
\end{equation}
for all $X \in \m$. Note that the first term on the right-hand side of \eqref{eq:derlivf} differs by the sign from that in the standard formula for the covariant derivative of Killing vector fields (e.g. \cite[Proposition~7.28]{Bes}); equation~\eqref{eq:derlivf} easily follows from that formula and the fact that the Lie bracket of a ($G$-)Killing vector field and a $G$-left-invariant vector field (as of vector fields on $M$) vanishes.

The elements $N_1, N_2, \dots$ of the Frenet frame at $o$ are orthonormal unit vectors in $\f$. From \eqref{eq:tgf}, \eqref{eq:derlivf} and the fact that $k_1 \nu_1=\n_\tau \tau$ we have
\begin{equation}\label{eq:nu1}
\<T, [X, T]_\m\>=k_1\<N_1,X\>,
\end{equation}
for all $X \in \m$. From the Frenet equations we have $\nu_2=k_2^{-1}k_1 \tau+k_2^{-1}\n_\tau \nu_1$. Then for any $X \in \f$, we obtain $\<N_2,X\>= \frac12 k_2^{-1}(\<[T, N_1]_\m,X\>+\<T,[X, N_1]_\m\>+\<N_1, [X,T]_\m\>)= k_2^{-1}\<[T, N_1]_\m,X\>$, since $\g_1$ is a subalgebra and by \eqref{eq:tau}, \eqref{eq:tgf}, \eqref{eq:derlivf}. As $\<\tau, \nu_2\>=0$ we get from \eqref{eq:nu1}
\begin{equation}\label{eq:nu2}
N_2= k_2^{-1}[T, N_1]_\m+k_2^{-1}k_1T.
\end{equation}

By a classical result \cite{Lie, T, Hof}, a subalgebra $\g_1 \subset \g$ of codimension one must contain the kernel $\ig$ of a homomorphism from $\g$ to $\slg(2)$. Denote $\ig'$ the (linear) projection of $\ig$ to $\m$. As $\ig \subset \g_1$, we have $\ig' \subset \f$, so $T \perp \ig'$. Moreover, since $\ig' \subset \ig + \h$ we get by \eqref{eq:tau} $[T, \ig'] \subset [T,\ig] \subset \ig$, as $\ig$ is an ideal. It follows that
$[T, \ig']_\m \subset \ig' \subset \f$. Taking $X \in \ig'$ in \eqref{eq:nu1} we then get $N_1 \perp \ig'$. Furthermore, taking the inner product of \eqref{eq:nu2} with $X \in \ig'$ we get $\<N_2, X\> = -k_2^{-1}\<[T, X]_\m,N_1\>$ by \eqref{eq:tgf}. But from the above, $[T, \ig']_\m \subset \ig'$ and $N_1 \perp \ig'$, so $N_2 \perp \ig'$. Therefore the codimension of $\ig$ in $\g$ is at least three (since $\ig \subset \ig' \oplus \h \subset (\m \cap \Lg^\perp) \oplus \h$), hence it is exactly three, with $\g/\ig\simeq\slg(2)$ and $\ig'=\m \cap \Lg^\perp$ and $\ig=\ig' \oplus \h$. In particular, $\h \subset \ig$. This proves \textbf{assertion~(\ref{it:helix2i})}, with $\mathfrak{I}=\ig$, and \textbf{assertion~(\ref{it:helix2ii})} (in view of \eqref{eq:tau}).

Now from \eqref{eq:tgf}, \eqref{eq:derlivf} we obtain $\<(\n_\tau \nu_2)(o), X\>= -\<[T,X]_\m, N_2\> - \frac12\<[N_2,X]_\m, T\>$, for any $X \in \m$. The right-hand side vanishes for all $X \in \ig'$, and also for $X = N_2$ and $X = T$ (from \eqref{eq:nu1} or from Frenet equations). It follows that $(\n_\tau \nu_2)(o)=-k_2 N_1$, hence $\n_\tau \nu_2(o)=-k_2 \nu_1$, which proves \textbf{assertion~(\ref{it:helix1})}.

Note that for any $X, Y \in \m$ with $[X,\h]=[Y,\h]=0$ we have $[\n_XY,\h]=0$. Indeed, for an orthonormal basis $e_i$ for $\m$ we have $\n_XY=\frac12 \sum_i(\<[X,Y]_\m,e_i\>+\<X,[e_i,Y]_\m\>+\<Y,[e_i,X]_\m\>)e_i$, so $[Z, \n_XY]= \frac12 [Z,[X,Y]_\m]+\frac12 \sum_i (\<X,[e_i,Y]_\m\> + \<Y,[e_i,X]_\m\>)[Z,e_i]$. For the first term on the right-hand side we have: $[Z,[X,Y]_\m]=[Z,[X,Y]]_\m=0$, where the first equality follows from the fact that both $\m$ and $\h$ are $\ad_\h$-invariant, and the second, from the Jacobi identity. As $[Z, \m] \subset \m$ and as $\ad_Z$ is skew-symmetric on $\m$, we obtain for the second term on the right-hand side (the third one is treated similarly): $\sum_i \<X,[e_i,Y]_\m\>[Z,e_i] = \sum_{i,j} \<X,[e_i,Y]_\m\>\<[Z,e_i],e_j\>e_j = -\sum_{i,j} \<X,[e_i,Y]_\m\>\<[Z,e_j],e_i\>e_j = - \sum_j \<X,[[Z,e_j],Y]_\m\>e_j = \sum_j \<X,[[e_j,Y],Z]_\m\>e_j$, by the Jacobi identity. But $[[e_j,Y],Z]_\m = [[e_j,Y]_\m,Z]$, so $\<X,[[e_j,Y],Z]_\m\>=-\<[e_j,Y]_\m,[X,Z]\>=0$. So
 $[\n_XY,\h]=0$.

Therefore, as $T$ commutes with $\h$, the vectors $N_1$ and $N_2$ also do. This proves \textbf{assertion~(\ref{it:helix2iii})}.

(c) The subspace $\Lg$ with the inner product induced from $\m$ is spanned by the orthonormal vectors $T, N_1, N_2$. The fact that $\g/\ig=\slg(2)$ follows from the above. Explicitly, for $X, Y \in \Lg$ denote $[X,Y]_{\Lg}$ the orthogonal projection of $[X,Y]_\m$ to $\Lg$. Then from \eqref{eq:nu2} we get $[T, N_1]_\Lg = k_2N_2-k_1T$. Moreover, from \eqref{eq:tgf}, $\<[T, N_2]_\Lg, N_2\>=0$ and $\<[T, N_2]_\Lg, N_1\>=-\<[T, N_1]_\Lg, N_2\>=-k_2$, and from \eqref{eq:nu1} $\<[T, N_2]_\Lg, T\>=0$. It follows that $[T, N_2]_\Lg = -k_2N_1$. Furthermore, as $N_1, N_2 \in \g_1$, the bracket (in $\g$) also lies in $\g_1$, so $[N_1, N_2]_\Lg \in \Span(N_1, N_2)$. From the Jacobi identity it then follows that $[N_1, N_2]_\Lg=-k_1 N_2$. The isometric isomorphism between the metric Lie algebras $\Lg$ and $\slg(2)$, with the inner product \eqref{eq:sl2inner}, is given by the correspondence $T=E_1, \; N_1=-E_3, \; N_2 = E-2$ and $k_1=2b, \; k_2=2a$.
\end{proof}

As the leaves tangent to $D$ are congruent helices of the second order, we will use a more conventional notation for their curvature and torsion:  $k=k_1$ and $\kappa=k_2$ respectively.

We can now introduce analytic local coordinates $t, u^1, \dots , u^{n-1}$ in a neighbourhood of any point $x \in M$ in such a way that $D= \Span (\partial/\partial t), \; D^\perp= \Span (\partial/\partial u^i \, : \, i=1, \dots, n-1), \; t(x)=u^i(x)=0$, and the leaf of $D$ passing through $x$ is parametrised by the arclength. Then the metric of $M$ in a neighbourhood of $x$ is given by $ds^2= e^{2\phi(u,t)} dt^2 + B_{ij}(u,t) du^i du^j$, where $\phi$ is an analytic function with $\phi(0,t)=0$ and $B$ is analytic and positively definite. We have
\begin{equation}\label{eq:frenet}
\begin{gathered}
\tau= e^{-\phi}\frac{\partial}{\partial t}, \qquad k \nu_1= \n_\tau \tau = - B^{ij} \frac{\partial \phi}{\partial u^i} \frac{\partial}{\partial u^j}, \qquad k^2= B^{ij} \frac{\partial \phi}{\partial u^i} \frac{\partial \phi}{\partial u^j} = \mathrm{const} \ne 0,\\
k \kappa \nu_2 = \n_\tau(k \nu_1) + k^2 \tau = -e^{-\phi} B^{ij} \frac{\partial^2 \phi}{\partial u^i \partial t} \frac{\partial}{\partial u^j}, \qquad \n_\tau(k \kappa \nu_2) = -e^{-\phi} B^{ij} \frac{\partial}{\partial t}\Big(e^{-\phi} \frac{\partial^2 \phi}{\partial u^i \partial t}\Big)  \frac{\partial}{\partial u^j} \, .
\end{gathered}
\end{equation}
As the leaves tangent to $D$ are helices of order two, we have from Frenet equations: $\n_\tau(k \kappa \nu_2) = -k \kappa^2 \nu_1$, so by \eqref{eq:frenet}, $\frac{\partial}{\partial t}\big(e^{-\phi} \frac{\partial^2 \phi}{\partial u^i \partial t}\big) = - \kappa^2 e^{\phi} \frac{\partial \phi}{\partial u^i}$, for all $i=1, \dots, n-1$, which gives $\frac{\partial}{\partial u^i}\big(\frac{\partial^2 \phi}{\partial t^2}-\frac12 (\frac{\partial \phi}{\partial t})^2 + \frac12 \kappa^2 e^{2\phi}\big)=0$. It follows that $\frac{\partial^2 \phi}{\partial t^2}-\frac12 (\frac{\partial \phi}{\partial t})^2 + \frac12 \kappa^2 e^{2\phi}= h(t)$, for some analytic function $h$. But $\phi(0,t)=0$, so $h(t) = \frac12 \kappa^2$. This gives the equation $\frac{\partial}{\partial t}\big(e^{-\phi}(\frac{\partial \phi}{\partial t})^2+\kappa^2 (e^{\phi}+e^{-\phi})\big) =0$. Solving this equation we get $e^{-\phi}=\sinh(\a(u))\cos(\kappa t+\b(u))+\cosh(\a(u))$ for some analytic functions $\a, \b$ on a neighbourhood of $x \in M_2$. This gives the required expression for the twisting function in \textbf{Case~(\ref{it:th3}) of Theorem~\ref{t:tg}}. A direct calculation using the fact that $k^2= B^{ij} \frac{\partial \phi}{\partial u^i} \frac{\partial \phi}{\partial u^j}$ from \eqref{eq:frenet} shows that $\|\n \a \|^2 = \sinh(\a)^2 \|\n\b\|^2= k^2$, as required.
%
%This completes the proof of Theorem~\ref{t:tg}.
\end{proof}

As we can see from the proof, in Case~\eqref{it:th3}, the subgroup $G_1 \subset G$ acts transitively on $F$, so $F$ is a homogeneous totally geodesic hypersurface. To some surprise, there exist non-homogeneous totally geodesic hypersurfaces belonging to Case~\eqref{it:th2}, as the following example shows. %(probably the simplest?)

\begin{example}\label{ex:nonhomo}
Consider the metric solvable Lie algebra with an orthonormal basis $Z, X_1, X_2, Y$ and with the nonzero brackets $[Z,X_1]=X_1+X_2, \; [Z,X_2]=-X_1+X_2, \; [Z,Y]=2Y$. The corresponding left-invariant metric on $\br^4$ is given by $ds^2=dz^2+e^{4z}dy^2+e^{2z}(dx_1^2+dx_2^2)$, with $Z=-\frac{\partial}{\partial z}, \; Y=e^{-2z}\frac{\partial}{\partial y}, \; X_1=e^{-z}\cos z \frac{\partial}{\partial x^1}-e^{-z}\sin z \frac{\partial}{\partial x^2}, \; X_2=e^{-z}\sin z \frac{\partial}{\partial x^1}+e^{-z}\cos z \frac{\partial}{\partial x^2}$. The resulting homogeneous space indeed belongs to Case~\eqref{it:th2}, with $M_2=\{x_1=x_2=0\}$, the hyperbolic space. Moreover, the hypersurface $F=\{x_1=0\}$ is totally geodesic, but $\Span(Z,Y,X_2)$ is not a subalgebra.
\end{example}

\begin{proof}[Proof of Theorem~\ref{t:subm}]
Suppose the pair $(M=G/H, F)$ belongs to Case~\eqref{it:th3} of Theorem~\ref{t:tg}. In the notation of Lemma~\ref{l:helix}, let $N$ be the connected (normal) subgroup of $G$ tangent to the ideal $\mathfrak{I}$. Then by Lemma~\ref{l:helix} \eqref{it:helix2i} $N \supset H$ and $G/N=\widetilde{\mathrm{SL}(2)}$ (as the Lie group). Moreover, from Lemma~\ref{l:helix} \eqref{it:helix3} it follows that the projection $\pi: M \to \widetilde{\mathrm{SL}(2)}$ (defined by $\pi(gH)=gN$ for $g \in G$), where  $\widetilde{\mathrm{SL}(2)}$ is equipped with the left-invariant metric defined by the inner product \eqref{eq:sl2inner} (with some specific choice of the constants $a$ and $b$), is a Riemannian submersion. Then the projection of $F$ to $\widetilde{\mathrm{SL}(2)}$ is the connected Lie subgroup $G_1/N$ whose Lie algebra is spanned by $\sg$.
\end{proof}

From Theorem~\ref{t:subm} (or from Lemma~\ref{l:helix}) it follows that Case~\eqref{it:th3} of Theorem~\ref{t:tg} may only occur if the semisimple part of the Levi-Mal'cev decomposition of the Lie algebra of (any) transitive group of isometries of $M$ contains an ideal isomorphic to $\slg(2)$.

% vklyuchat' eto? then need characterisation of (b) (exists ideal of codim 1?
%\begin{remark*}
%Note that the Corollary from Section~\ref{s:intro} can be sharpened. From Theorem~\ref{t:subm} (or from Lemma~\ref{l:helix}) it follows that if a simply connected homogeneous space $M$ admits a totally geodesic hypersurface of Case~\eqref{it:th3} and if $G$ is any transitive group of isometries of $M$, then the semisimple part of the Levi-Mal'cev decomposition of the Lie algebra of $G$ must contain an ideal isomorphic to $\slg(2)$.  of imposes . % perenesti v konec?

%\begin{corollary}\label{cor:com}
%Suppose $M$ is a simply connected, connected Riemannian homogeneous space and $F \subset M$ is a complete connected totally geodesic hypersurface. In each of the following cases, $M$ is the direct product of a space $M_1(c)$ of constant curvature and a homogeneous space $M_2$ and $F=F_1(c) \times M_2$, where $F_1(c) \subset M_1(c)$ is a totally geodesic hypersurface:

%\begin{enumerate}
%  \item $M$ is compact.
%  \item $M$ admits a transitive group of isometries $G$ whose Lie algebra $\g$ contains no ideals isomorphic to $\slg(2)$ in the semisimple part of its Levi-Mal'cev decomposition and no ideals of codimension one.
%\end{enumerate}
%\end{corollary}

%Note that even when $\g$ does contain a codimension one ideal, $M$ may still split as a direct product; in particular, this happens when $M$ is a %nilmanifold \cite[Proposition~1.13.]{CHN1}.

%\end{remark*}

% additional requirements for \a and \b?


\begin{thebibliography}{CHN2}

\bibitem[AN]{AN}
Alekseevsky D., Nikonorov Y.,
\emph{Compact Riemannian manifolds with homogeneous geodesics},
SIGMA Symmetry Integrability Geom. Methods Appl. \textbf{5} (2009), Paper 093, 16 pp.

\bibitem[Bes]{Bes}
Besse A.,
\emph{Einstein manifolds},
Springer, Berlin, Heidelberg, New York, 1987.

\bibitem[BH]{BH}
Blumenthal R., Hebda J.
\emph{de Rham decomposition theorems for foliated manifolds},
Ann. Inst. Fourier (Grenoble) \textbf{33} (1983), 183 -- 198.

\bibitem[CHN1]{CHN1}
Cairns G., Hini\'{o}-Gali\'{c} A., Nikolayevsky Y.,
\emph{Totally geodesic subalgebras of nilpotent Lie algebras. I}, preprint, arXiv:1112.1288.

\bibitem[CHN2]{CHN2}
Cairns G., Hini\'{o}-Gali\'{c} A., Nikolayevsky Y.,
\emph{Totally geodesic subalgebras of nilpotent Lie algebras. II}, preprint, arXiv:1112.1456.

\bibitem[CN]{CN}
Chen B.Y., Nagano T.
\emph{Totally geodesic submanifolds of symmetric spaces II},
Duke Math. J. \textbf{45} (1978), 405 -- 425.

\bibitem[C]{C}
Cartan, \'{E}.
\emph{Sur une classe remarquable d'espaces de Riemann}, Bull. Soc. Math. France \textbf{55} (1927), 114 -- 134.

\bibitem[Dus]{Dus}
Du\v{s}ek Z.,
\emph{The existence of homogeneous geodesics in homogeneous pseudo-Riemannian and affine manifolds},
J. Geom. Phys. \textbf{60} (2010), 687 -- 689.

\bibitem[Ebe]{Ebe}
Eberlein P.
\emph{Geometry of $2$-step nilpotent groups with a left invariant metric. II},
Trans. Amer. Math. Soc. \textbf{343} (1994), 805 -– 828.

\bibitem[Gor]{Gor}
Gordon C.,
\emph{Homogeneous Riemannian manifolds whose geodesics are orbits},
Topics in Geometry: in Memory of Joseph D'Atri, Progr. Nonlinear Differential Equations Appl., Vol. 20, Birkhäuser, Boston, MA, 1996, 155 -- 174.

\bibitem[Hel]{Hel}
S. Helgason, \emph{Differential geometry, Lie groups, and symmetric spaces}.
Pure and Applied Mathematics, 80. Academic Press, Inc. New York--London, 1978.

\bibitem[Hof]{Hof}
Hofmann K.,
\emph{Hyperplane subalgebras of real Lie algebras},
Geom. Dedicata \textbf{36} (1990), 207 -- 224.

\bibitem[JMS]{JMS}
Jentsch T., Moroianu A., Semmelmann U.,
\emph{Extrinsic hyperspheres in manifolds with special holonomy}, preprint, arXiv:1107.1603.

\bibitem[Kai]{Kai}
Ka\u{i}zer V.,
\emph{Conjugate points of left-invariant metrics on Lie groups},
Soviet Math. (Iz. VUZ) \textbf{34} (1990), 32 -- 44.

\bibitem[KP]{KP}
Kerr M., Payne T.,
\emph{The geometry of filiform nilpotent Lie groups},
Rocky Mountain J. Math. \textbf{40} (2010), 1587 -- 1610.

\bibitem[KS]{KS}
Kowalski O., Szenthe J.,
\emph{On the existence of homogeneous geodesics in homogeneous Riemannian manifolds},
Geom. Dedicata \textbf{81} (2000), 209 -- 214.

\bibitem[Lie]{Lie}
Lie S.,
\emph{Theorie der Transformationsgruppen},
Math. Ann. \textbf{16} (1880), 441 -- 528.

\bibitem[T]{T}
Tits J.,
\emph{Sur une classe de groupes de Lie r\'{e}solubles},
Bull. Soc. Math. Belg. \textbf{11} (1959), 100 -- 115.

\bibitem[To1]{To1}
Tojo K.,
\emph{Totally geodesic submanifolds of naturally reductive homogeneous spaces},
Tsukuba J. Math. \textbf{20} (1996), 181 -- 190.

\bibitem[To2]{To2}
Tojo K.,
\emph{Normal homogeneous spaces admitting totally geodesic hypersurfaces},
J. Math. Soc. Japan, \textbf{49} (1997),  781 -- 815.

\bibitem[Ts1]{Ts1}
Tsukada K.,
\emph{Totally geodesic hypersurfaces of naturally reductive homogeneous spaces},
Osaka J. Math. \textbf{33} (1996), 697 -- 707.

\bibitem[Ts2]{Ts2}
Tsukada K.,
\emph{Totally geodesic submanifolds of Riemannian manifolds and curvature-invariant subspaces},
Kodai Math. J. \textbf{19} (1996), 395 -- 437.

\end{thebibliography}
\end{document}